\documentclass[11pt]{article}
\usepackage{amsmath,amssymb,amstext,dsfont,fancyvrb,float,fontenc,graphicx,subfigure,theorem,hyperref}
\usepackage[utf8]{inputenc}

\usepackage{tikz,everypage}

\usepackage[letterpaper]{geometry}
\newfont{\footsc}{cmcsc10 at 8truept}
\newfont{\footbf}{cmbx10 at 8truept}
\newfont{\footrm}{cmr10 at 10truept}

\usepackage{fancyhdr}
\usepackage{relsize}
\usepackage{sectsty}
\allsectionsfont{\larger[-1]} 

\renewcommand\paragraph{\@startsection{paragraph}{4}{\z@}
                                    {2ex \@plus.5ex \@minus.2ex}
                                    {-1em}
                                    {\normalfont\normalsize\bfseries}}

\renewcommand\subparagraph{\@startsection{subparagraph}{5}{\parindent}
                                       {2ex \@plus.5ex \@minus .2ex}
                                       {-1em}
                                      {\normalfont\normalsize\bfseries}}

\newlength{\BiblioSpacing}
\renewenvironment{thebibliography}[1]{
\begin{oldthebibliography}{#1}
\setlength{\parskip}{\BiblioSpacing}
\setlength{\itemsep}{\BiblioSpacing}
}
{
\end{oldthebibliography}
}

\usepackage[strict]{changepage}
\def\abstractname{Abstract -}   % <-----------------
\def\abstract{\begin{adjustwidth}{1cm}{1cm} \par    \footnotesize \noindent {\bf \abstractname} 
\def\endabstract{ \end{adjustwidth} \smallskip }}
{\theorembodyfont{\itshape}\newtheorem{theorem}{Theorem}[section]}
{\theorembodyfont{\itshape}}
{\theorembodyfont{\itshape}\newtheorem{definition}[theorem]{Definition}}
{\theorembodyfont{\itshape}\newtheorem{lemma}[theorem]{Lemma}}
{\theorembodyfont{\itshape}}
{\theorembodyfont{\rm}}
{\theorembodyfont{\rm}}
{\theorembodyfont{\rm}}
{\theorembodyfont{\rm }}
\def\proof{{\noindent \bf Proof.~}}
\def\endproof{\hspace{\fill}{\boldmath $\Box$}\par}

\usepackage{tikz}
\usepackage{pgfplots}

\makeatletter
\pgfmathdeclarefunction{weierstrass}{4}{%
    \pgfmathfloattofixed@{#4}%
    \afterassignment\pgfmath@x%
    \expandafter\c@pgfmath@counta\pgfmathresult pt\relax%
    \pgfmathfloatcreate{1}{0.0}{0}%
    \let\pgfmathfloat@loc@TMPr=\pgfmathresult
    \pgfmathfloatpi@%
    \let\pgfmathfloat@loc@TMPp=\pgfmathresult%
    \edef\pgfmathfloat@loc@TMPx{#1}%
    \edef\pgfmathfloat@loc@TMPa{#2}%
    \edef\pgfmathfloat@loc@TMPb{#3}%
    \pgfmathloop
        \ifnum\c@pgfmath@counta>-1\relax%
            \pgfmathfloatparsenumber{\the\c@pgfmath@counta}%
            \let\pgfmathfloat@loc@TMPn=\pgfmathresult%
            \pgfmathpow{\pgfmathfloat@loc@TMPa}{\pgfmathfloat@loc@TMPn}%
            \let\pgfmathfloat@loc@TMPe=\pgfmathresult%
            \pgfmathpow{\pgfmathfloat@loc@TMPb}{\pgfmathfloat@loc@TMPn}%
            \pgfmathmultiply{\pgfmathresult}{\pgfmathfloat@loc@TMPp}%
            \pgfmathmultiply{\pgfmathresult}{\pgfmathfloat@loc@TMPx}%
            \pgfmathdeg{\pgfmathresult}%
            \pgfmathcos{\pgfmathresult}%
            \pgfmathmultiply{\pgfmathresult}{\pgfmathfloat@loc@TMPe}%
            \pgfmathadd{\pgfmathresult}{\pgfmathfloat@loc@TMPr}%
            \let\pgfmathfloat@loc@TMPr=\pgfmathresult
            \advance\c@pgfmath@counta by-1\relax%
    \repeatpgfmathloop%
}

\usepackage{xcolor}
\usepackage{hyperref}
\hypersetup{
    colorlinks=true, %set true if you want colored links
    linktoc=all,     %set to all if you want both sections and subsections linked
    linkcolor=blue,  %choose some color if you want links to stand out
}

\title{Upper bounds for the Hausdorff dimension of Weierstrass curves.}
\author{Ted Alexander and Tommy Murphy }
\date{}

\begin{document}
\setcounter{page}{1}
%\date{}
%%%%%%%%%%%%%%%%%%%%%%%%%%%%%%%%%%%%%%%%%
\maketitle
\thispagestyle{fancy}

\vskip 1.5em

\begin{abstract}
We produce an upper bound for the Hausdorff dimension of the graph of a Weierstrass-type function. Whilst strictly weaker than existing results, it has the advantage of being directly computable from the theory of hyperbolic iterated function systems (IFS).
\end{abstract}

\section{Introduction}

The concept of Hausdorff dimension is  intimately bound up with the study of fractals; for instance Mandelbrot's well-known assertion that a fractal is characterized by the  Hausdorff dimension strictly exceeding the topological dimension. In some special cases it is easy to compute. The Koch curve  can be defined using four translated contractions of itself. Since the scaling factor is $1/3$ and four non-warping contractions are used, the Hausdorff dimension of the Koch Curve is exactly $\log_34$. In contrast, it is a remarkable fact that for arguably the earliest known example of a fractal, namely Weierstrass' monster,  the precise computation of the Hausdorff dimension was an open problem until quite recently. The essential difficulty is that the contraction mappings defining the fractal have uneven warping, which complicates matters significantly.  

The Hausdorff dimension of the graph of the function
 \[
 W_{a, b}(x) = \sum_{n=0}^{\infty} a^n\cos(2\pi b^nx)
 \]
 where $x\in \mathbb{R}$, $b\in \mathbb{N}$, and  $\frac{1}{b}< a < 1$, was long conjectured to be 
$D : = 2 + \log_ba.$ This was  settled  by  Shen \cite{s} in 2018. The classical examples of Weierstrass were of the form $b\in \mathbb{N}$ and $ab + 1> \frac{3\pi}{2}$. These became famous in the mathematical world as they were the first published examples of functions which are everywhere continuous yet nowhere differentiable.

\begin{figure}[!htb]
\begin{center}
\begin{tikzpicture}
\begin{axis}[axis lines=middle, axis equal image, enlarge y limits=true]
\addplot [thick, black, samples=301, line join=round, domain=-2:2] {weierstrass(x,0.5,3,10)};
\end{axis}
\end{tikzpicture}
\caption{A classical Weierstrass curve}
\end{center}

\end{figure}
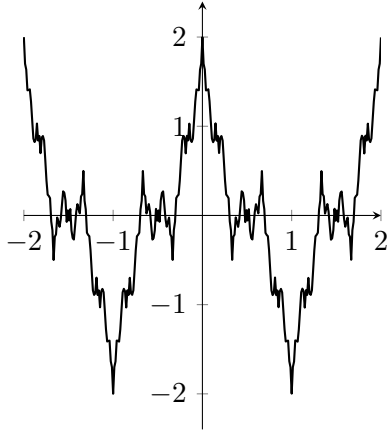

Throughout this paper $b\in \mathbb{N}$ and $\frac{1}{b}< a < 1$. Let $\phi$ be a $C^1$ function defined on $[0,1]$, and also denote by $\phi$ its $\mathbb{Z}$-periodic extension to $\mathbb{R}$. Set 
 \[
 w^{\phi}_{a,b}(x) = \sum_{n=0}^{\infty} a^n\phi(b^nx).
 \]
If $F\subset \mathbb{R}^2$, let $\dim_{H}[F]$ denote the Hausdorff dimension. Given a function $w: \mathbb{R}\rightarrow \mathbb{R}$, $\text{graph}(w)\subset \mathbb{R}^2$ will denote the graph of the function.   We can now state Shen's Theorem:
 
\begin{theorem}(Shen) There exists a $K_0 =K_0(\phi,b)>1$ such that if $1<ab<K_0$, then 
\[
\dim_{H}[ \textrm{graph}( w^{\phi}_{a,b})] = D.
\]
\end{theorem}
Clearly the classical examples of Weierstrass follow on setting $\phi(x) = \cos(2\pi x)$ and choosing $a$ and $b$ appropriately.  Since $D>1$, this in particular produces many examples of fractals. 

Shen's work is the culmination of many years of research beginning with the work of Bescovitch-Ursell \cite{bu}.  Of particular interest to us is the well-known estimate 
\begin{equation}\label{ubd}
\dim_{H}[ \textrm{graph}( w^{\phi}_{a,b})] \leq D.
\end{equation} 
The argument to establish this is standard, but  indirect; see Section 2 of \cite{ba} for details.  One uses the fact $\dim_{H}[F]\leq \dim_B[F]$, where $\dim_B$ denotes the box-counting dimension and $F\subset \mathbb{R}^2$. The box dimension of $\textrm{graph}(w)$ is then estimated via  studying local oscillations in terms of H\"older continuity. Consequently, the main question in the field has been to understand lower bounds for the Hausdorff dimension, and Shen's theorem answers this question for a wide family of examples. We also refer the interested reader to the related works \cite{bbr}, \cite{hl}, \cite{hunt}, and \cite{t}.

Our result is the following. 
\begin{theorem}\label{maint}
If $|\phi'(x)|\leq 1$ and $a^2 + \frac{a}{b}<1$,   then \[
\dim_{H}[ \textrm{graph} ( w^{\phi}_{a,b})] \leq \log_h(1/b),
\]
where
\[
	h = \sqrt{\frac{1}{2}\left[\frac{2}{b^2}+a^2+\sqrt{\frac{4}{b^4}+a^4}\right]}.
\]
\end{theorem}
It is not hard to see that $\log_h (1/b)> D$; write $D=\log_b(b^2a)$ and change the base of the logarithm on the right hand side.  This means our bound is always worse than Equation (\ref{ubd}).  We already know  that this must be the case since Shen's Theorem states that $D$ actually {\textit{is}} the Hausdorff dimension for a wide class of examples. The merit of our result is that it avoids estimating the Hausdorff dimension via the approach of estimating  the box-counting dimension, but rather uses the theory of iterated function systems (IFS). Our main technical achievement is the observation that there is a global upper contraction bound on the IFS determined by $\phi$ under our assumptions. 

Many standard examples of IFS are given by linear transformations, written as $2\times 2$ matrices with constant coefficients (see \cite{bd2} for many such examples). The techniques of our proof will also apply in these instances. The case of Weierstrass curves was more interesting to us as the coefficients of the matrix vary, so estimating the contraction factors is harder. We would expect further examples of fractals could also be analyzed in this framework.

\section{Setup}
Here we set up some basic notation and definitions. Throughout we work in the standard metric space $W= [0,1]\times \mathbb{R}\subset \mathbb{R}^2$. Standard texts explaining the basics of IFS are \cite{bd} and \cite{f}, following the foundational work of Hutchinson \cite{h}.  Let $\{S_i\}_{i=1}^b$ be  contraction mappings on $W$ with contraction factors $\{u_i\}_{i=1}^b$.  The class of non-empty compact subsets of $W$ equipped with the associated Hausdorff metric then has associated contraction mappings, also denoted $S_i$, with the same contraction factors.   Let $F$ be the invariant set for $\{S_i\}$, i.e.
   \[
 F=\bigcup_{i=1}^bS_i(F).
\]
 The basic idea underlying the theory of IFS \cite{bd} is that the existence and uniqueness of  $F$ is granted by the Banach fixed-point Theorem.

\begin{definition} Given a set $F\subset \mathbb{R}^2$ with $\delta$-covers $U_i$, we define the \underline{Hausdorff $s$-content} $H^s(F)$ to be $H^s(F)=\inf \sum_i |U_i|^s$,
where the infimum is taken over all such possible $\delta$-covers. The Hausdorff dimension $\dim_{H}(F)$ is  defined to be the infimal positive $s$ such that $H^s(F)$ is finite.
\end{definition}

\begin{lemma}\label{l1} $\dim_H[F]\leq s$, where  $\sum_{i=1}^bu_i^s=1$.
 
\end{lemma}
\proof See Theorem 8.8/Exercise 8.5 of \cite{f}. The open set condition  required is satisfied taking $V$ to be a small open tubular neighbourhood of  $F\backslash \lbrace x = 0, 1\rbrace$. \endproof

\section{Contraction mappings associated to Weierstrass curves IFS}

Endow $W\subset \mathbb{R}^2$ with its usual metric space structure. It is standard \cite{ba} to rewrite the graph of a Weierstrass curve as an IFS using the mappings
\begin{equation}\label{con1}
S_i(x,y) = \left(\frac{x+i-1}{b}, ay + \phi\left(\frac{x+i-1}{b}\right)\right)\ \ \ \ \ \ \ \ \ \ 1\leq i \leq b.
\end{equation}
There are various related definitions of an IFS in the literature. Our definition, following \cite{bd}, is sometimes referred to as a  hyperbolic IFS: each $S_i$ is a contraction mapping. In \cite{ba} and \cite{bbr}, Equation (\ref{con1}) defines a smooth nonlinear system with two negative Lyapunov exponents which they also call an IFS. This is a little different to our definition because the mappings (\ref{con1}) are not assumed to be contraction mappings. However, under our additional assumptions each $S_i$ is a contraction mapping and so we can apply some standard techniques to bound the Hausdorff dimension.

\begin{lemma}\label{l2}  Under the assumptions of  Theorem \ref{maint}, each $S_i$ is a contraction mapping. 
\end{lemma}
\begin{proof}Choose distinct points ${\bf{x}_1} = (x_1, y_1)$ and ${\bf{x}_2} =  (x_2, y_2)$ in $W$.  We need to show that
\begin{equation}\label{contr1}
d(S_i({\bf{x}_1}), S_i({\bf{x}_2})) < d({\bf{x}_1},{\bf{x}_2}).
\end{equation}
The left-hand side is
$$
d(S_i({\bf{x}_1}), S_i({\bf{x}_2}))  = \sqrt{ \left[\frac{\Delta x}{b}\right]^2 + \left[ a\Delta y + \phi\left(\frac{x_1 + i}{b}\right) - \phi\left(\frac{x_2 + i}{b}\right)\right]^2}
$$
where $\Delta x = x_1-x_2$ and $\Delta y = y_1-y_2$. Since $\phi\in C^1$, applying the mean value theorem there is a positive number $c <1$ so that 
$$
\bigg| \phi\left(\frac{x_1 + i-1}{b}\right) - \phi\left(\frac{x_2 + i-1}{b}\right)\bigg| = \frac{c}{b}|\Delta x|.
$$
Plugging this in and expanding, Equation (\ref{contr1}) beomes
\begin{equation}\label{22}
\sqrt{ \frac{1 + c^2}{b^2} (\Delta x)^2 + a^2(\Delta y)^2 + \frac{2ac}{b} \Delta x \Delta y} < \sqrt{ (\Delta x)^2 + (\Delta y)^2}.
\end{equation}

Applying the AM-GM inequality, 
\begin{equation}\label{e3}
\bigg|\frac{2ac}{b} \Delta x \Delta y \bigg| \leq \frac{ac}{b} \bigg((\Delta x )^2 + (\Delta y)^2\bigg).
\end{equation}
Squaring both sides of Equation (\ref{22}), applying the triangle inequality and Equation (\ref{e3}),  and splitting  the $(\Delta x)^2$  and $(\Delta y)^2$ terms, we see Equation (\ref{contr1}) will follow if we show that
\begin{equation}
\frac{ 1 + c^2}{b^2} + \frac{ac}{b} < 1 \ \ \ \ \text{and} \ \ \ \ a^2+ \frac{ac}{b} <1.  
\end{equation}
Noting that the left-hand side of both of these inequalities is an increasing function of $c$, which is the value of the derivative of $\phi$ at some point, and $|\phi' |\leq1$, we see $c\leq1$ which leads to
\begin{equation}
\frac{ 2}{b^2} + \frac{a}{b} < 1 \ \\ \ \ \text{and}  \ \ \ a^2+ \frac{a}{b} <1.  
\end{equation}
The first equation always holds, since $|a|<1$ and $b\in \mathbb{N} >1$. The second equation holds as that is precisely the assumption on the coefficients in the statement of the main theorem. 
\end{proof}

\section{Proof of Theorem \ref{maint}}

Armed now with the knowledge that our the mappings $S_i$ are contraction mappings, the strategy of our proof is   to apply Lemma \ref{l1} to estimate the Hausdorff dimension.

 \begin{proof}
 From Lemma \ref{l1}, it is clear that we need to estimate the contraction factors $u_i$.  Following the lines of the proof of Lemma \ref{l2}, choose  distinct points ${\bf{x}_1} = (x_1, y_1)$ and ${\bf{x}_2}= (x_2, y_2) \in W$. Then $d^2(S_i( {\bf{x}_1}), {S_i(\bf{x}_2} ))$ can be written in matrix form as
\begin{equation}
\left(\begin{array}{cc} \Delta x & \Delta y\end{array}\right) \left(\begin{array}{cc} \frac{1 + c^2}{b^2} & \frac{ac}{b} \\
\frac{ac}{b} & a^2 \end{array}\right) \left(\begin{array}{c} \Delta x \\ \Delta y \end{array} \right).
\end{equation} 

Now view $v = (\Delta x, \Delta y)$ as an element of $\mathbb{R}^2$: the question is how to extremize  $\sqrt{v^TAv}$, where $T$ denotes the transpose and $A$ is the positive definite symmetric matrix
\begin{equation}\label{mat1}
\left(\begin{array}{cc} \frac{1 + c^2}{b^2} & \frac{ac}{b}\\
\frac{ac}{b} & a^2 \end{array}\right). 
\end{equation}
The alert reader will note that  this is not a matrix with constant coefficients, since $c$ is determined, via the Mean-Value Theorem, by $x_1$ and $x_2$ and so  ultimately depends upon ${\bf{x_1}}$ and ${\bf{x_2}}$. Our proof proceeds by fixing $c$, so that  Equation (\ref{mat1}) is regarded as  a fixed symmetric matrix $A$. It is a standard fact that a positive definite  symmetric matrix has positive real eigenvalues and that   $\sqrt{v^TAv} \leq \sqrt{\lambda}\|v\|$, where $\lambda$ denotes the largest eigenvalue of $A$.  We then vary obtain an upper bound that is independent of $c$.  For two distinct points  ${\bf{x_1}}, {\bf{x_2}}\in W$ there will be a corresponding  $c$ in the formula for $d(S_i( {\bf{x}_1}), {S_i(\bf{x}_2} ))$ and thus a corresponding matrix of the form (\ref{mat1}).  As our upper bound is independent of $c$ we can thus   estimate the contraction factor of $S_i$.

 A straightforward computation shows the eigenvalues of this matrix are is 
$$
    \lambda_{\pm} = \frac{1}{2}\left(\frac{1+c^2}{b^2}+a^2\right)\pm\frac{1}{2}\sqrt{\left(\frac{1+c^2}{b^2}+a^2\right)^2-\frac{4a^2}{b^2}}.
$$
If there is only one eigenvalue, $\left(\frac{1+c^2}{b^2}+a^2\right)^2-\frac{4a^2}{b^2} = 0$ which implies that $ab = 1 \pm \sqrt{-c^2}$, an immediate contradiction because $ab$ is real. So, there cannot be one repeated eigenvalue and hence there must be two distinct eigenvalues. For our purposes, we need only the larger eigenvalue to establish the upper bound. Hence we focus on 
\[
	\lambda = \frac{1}{2}\left(\frac{1+c^2}{b^2}+a^2\right)+\frac{1}{2}\sqrt{\left(\frac{1+c^2}{b^2}+a^2\right)^2-\frac{4a^2}{b^2}}.
\]
Note this is an increasing function of $c$. As $|\phi'|\leq1$, we set $c=1$ to obtain
\[
\lambda_{\text{max}} = \frac{1}{2}\left[\frac{2}{b^2}+a^2+\sqrt{\frac{4}{b^4}+a^4}\right].
\]
This directly implies an upper bound for the contraction factor for each $S_i$  is \[u_i = \sqrt{\lambda_{\text{max}}}:= h \ \ \ \ \ \  1\leq i \leq b.\]
 Hence, by Lemma \ref{l2} an upper bound on the Hausdorff dimension of the graph of $w$ is  given by solving $bh^s=1.$ Equivalently,
$$
    s =\log_{h}(1/b).$$
The result now follows. 
\end{proof}

\section*{Acknowledgments} 
 
We thank the Department of Mathematics at Cal State Fullerton for encouraging undergraduate research and for supporting T.A. with a summer research scholarship. T. M. thanks the mathematics department at  UC Irvine for their hospitality whilst this work was written up. Both authors thank K. Bara\'nski and the anonymous referee for helpful comments.

\newpage 

{\footnotesize

%%%%%%%% AUTHORS' INFORMATION. DELETE/ADD AUTHORS AS NEEDED
{\footnotesize  
\medskip
\medskip
\vspace*{1mm} 
 
\noindent {\it Tommy Murphy}\\  
Department of Mathematics,\\
CSU Fullerton,\\
800 N. State College Blvd.,\\
Fullerton CA 92831.\\
E-mail: {\tt tmurphy@fullerton.edu}\\  
\href{http://www.fullerton.edu/math/faculty/tmurphy/}{http://www.fullerton.edu/math/faculty/tmurphy/}\\

\noindent {\it Ted Alexander}\\  
Department of Mathematics,\\
CSU Fullerton,\\
800 N. State College Blvd.,\\
Fullerton CA 92831.\\
E-mail: {\tt tedforpresident@gmail.com}

%%%%%%%%%%%% Please do not remove or move the } sign below, do not remove blank line before it

}

\end{document}